\documentclass{amsart}

\usepackage{amssymb,stmaryrd,mathrsfs,dsfont,amsmath}
\usepackage{colonequals} %% For \colonequals as needed in the body. Compare := and := carefully!!!

\theoremstyle{plain}
\newtheorem{theorem}{Theorem}[section]
\newtheorem{lemma}[theorem]{Lemma}
\newtheorem{proposition}[theorem]{Proposition}
\newtheorem{corollary}[theorem]{Corollary}

\theoremstyle{definition}

\newtheorem{example}[theorem]{Example}

\theoremstyle{remark}
\newtheorem{remark}[theorem]{Remark}

\DeclareMathOperator{\dom}{\mbox{dom}}
\DeclareMathOperator{\codom}{\mbox{codom}}
\DeclareMathOperator{\rank}{\mbox{rank}}

\DeclareMathOperator{\ureg}{\mbox{ureg}}
\DeclareMathOperator{\codim}{\mbox{codim}}
\DeclareMathOperator{\nul}{\mbox{nullity}}
\DeclareMathOperator{\corank}{\mbox{corank}}

\input xy
\xyoption{all}

\begin{document}

\title[Unit-regular and semi-balanced elements in semigroups of transformations]
{Unit-regular and semi-balanced elements in various semigroups of transformations}

\author[Mosarof Sarkar]{\bfseries Mosarof Sarkar}
\address{Department of Mathematics, Central University of South Bihar, Gaya, Bihar, India}
\email{mosarofsarkar@cusb.ac.in}

\author[Shubh N. Singh]{\bfseries Shubh N. Singh}
\address{Department of Mathematics, Central University of South Bihar, Gaya, Bihar, India}
\email{shubh@cub.ac.in}

%\date{...}

\begin{abstract}
Let $T(X)$ be the full transformation semigroup on a set $X$, and let $L(V)$ be the semigroup under composition of all linear transformations on a vector space $V$ over a field.
For a subset $Y$ of $X$ and a subspace $W$ of $V$, consider the semigroups $\overline{T}(X, Y) = \{f\in T(X)\colon Yf \subseteq Y\}$ and $\overline{L}(V, W) = \{f\in L(V)\colon Wf \subseteq W\}$ under composition. We describe unit-regular elements in $\overline{T}(X, Y)$ and $\overline{L}(V, W)$. Using these, we determine when $\overline{T}(X, Y)$ and $\overline{L}(V, W)$ are unit-regular. We prove that $f\in L(V)$ is unit-regular if and only if $\nul(f) = \corank(f)$. We alternatively prove that $L(V)$ is unit-regular if and only if $V$ is finite-dimensional. A semi-balanced semigroup is a transformation semigroup whose all elements are semi-balanced. We give necessary and sufficient conditions for $\overline{T}(X, Y)$, $\overline{L}(V, W)$ and $L(V)$ to be semi-balanced.
\end{abstract}



\maketitle

\section{Introduction}
An element $s$ of a semigroup $S$ with identity is \emph{unit-regular} in $S$ if there exists a unit $u\in S$ such that $sus = s$.  A \emph{unit-regular semigroup} is a semigroup with identity in which every element is unit-regular.
The concept of unit-regularity, which was introduced by Ehrlich \cite{ehr68} within the context of rings, has consecutively received much attention from many semigroup theorists \cite{blyth83, chaiya-s19, chen-s74, alar-s80, fount02, hick97, mca76, mcfa84, tira-s79}. In fact, Alarcao \cite[Proposition 1]{alar-s80} proved that any semigroup with identity is unit-regular if and only if it is factorizable.

%%%%%%%%%%%%%%%%%%%%%%%%%%%%
%%%%%%%%%%%%%%%%%%%%%%%%%%%%
%%%%%%%%%%%%%%%%%%%%%%%%%%%%
%%%%%%%%%%%%%%%%%%%%%%%%%%%%
\vspace{0.1cm}

Let $X$ be a nonempty set. Denote by $T(X)$ the full transformation semigroup on $X$. The semigroup $T(X)$ and its subsemigroups are extremely important, since every semigroup can be embedded in some $T(Z)$ (cf. \cite[Theorem 1.1.2]{howie95}. In $1980$, Alarcao \cite[Proposition 5]{alar-s80} proved that $T(X)$ is unit-regular if and only if $X$ is finite. For a fixed nonempty subset $Y$ of $X$, Magill \cite{magill66} first studied  the semigroup $\overline{T}(X, Y)$ under composition consisting of all transformations $f\in T(X)$ such that $Y$ is invariant under $f$. Using symbols, 
\[\overline{T}(X, Y) = \{f\in T(X)\colon Yf \subseteq Y\}.\] 
If $Y=X$, then $\overline{T}(X, Y) = T(X)$. To this extent, the semigroup $\overline{T}(X, Y)$ may be regarded as a generalization of $T(X)$. Many interesting properties of the semigroup $\overline{T}(X, Y)$ and its certain subsemigroups have been investigated \cite{chinar-ae19, choom-13, hony11, nenth05, nenth06, shubh-aejm22, sun-as13, symon-as75}. In particular, Nenthein et al. \cite[Theorem 2.3]{nenth05} described
regular elements in $\overline{T}(X, Y)$.

%%%%%%%%%%%%%%%%%%%%%%%%%%%%
%%%%%%%%%%%%%%%%%%%%%%%%%%%%
%%%%%%%%%%%%%%%%%%%%%%%%%%%%
%%%%%%%%%%%%%%%%%%%%%%%%%%%%
%\vspace{0.1cm}

Let $V$ be a vector space over a field. Denote by $L(V)$ the semigroup under composition consisting of all linear transformations on $V$. Ehrlich \cite[Theorem 4]{ehr68} proved that $L(V)$ is not unit-regular when $V$ is infinite-dimensional. Jampachon et al. \cite[Theorem 2]{jamp-sa01} proved that $L(V)$ is factorizable if and only if $V$ is finite-dimensional. If $V$ is finite-dimensional, Kemprasit \cite{kamp-sa02} directly proved that $L(V)$ is unit-regular. For a fixed subspace $W$ of $V$, analogous to the semigroup $\overline{T}(X, Y)$, Nenthein and Kemprasit \cite{nenth06} introduced the semigroup $\overline{L}(V, W)$ under composition consisting of all linear transformations $f\in L(V)$ such that $W$ is invariant under $f$. Using symbols,
\[\overline{L}(V, W) = \{f\in L(V)\colon Wf \subseteq W\}.\] If $W$ is trivial, then $\overline{L}(V, W) = L(V)$. To this extent, the semigroup $\overline{L}(V, W)$ may be regarded as a generalization of $L(V)$. The authors \cite[Proposition 3.1]{nenth06} described regular elements in $\overline{L}(V, W)$ and proved that $\overline{L}(V, W)$ is regular if and only if $W$ is trivial. Chaiya \cite[Theorem 11]{chaiya-s19} determined when $\overline{L}(V, W)$ is a unit-regular semigroup.
Several other properties of the semigroup $\overline{L}(V, W)$ have also been studied \cite{chaiya-s19, chinar-im18, hony12, pei12}.

%%%%%%%%%%%%%%%%%%%%%%%%%%%%
%%%%%%%%%%%%%%%%%%%%%%%%%%%%
%%%%%%%%%%%%%%%%%%%%%%%%%%%%
%%%%%%%%%%%%%%%%%%%%%%%%%%%%
\vspace{0.5mm}
Let $f\in T(X)$. The \emph{collapse} (resp. \emph{defect}) of $f$ is the cardinality of $X\setminus T_f$ (resp. $X\setminus Xf$), where $T_f$ is any transversal of the equivalence relation $\ker(f)$ on $X$ defined by $(x,y)\in \ker(f)$ if $xf = yf$. In $1998$, Higgins et al. \cite[p. 1356]{hig-how-rus98} called an element of $T(X)$ to be \emph{semi-balanced} if its collapse and defect are equal, and denoted the set of all semi-balanced elements of $T(X)$ by $B$. We say that a subsemigroup of $T(X)$ is \emph{semi-balanced} if all its elements are semi-balanced. It is clear that every subsemigroup of $T(X)$ is semi-balanced when $X$ is finite. The authors \cite[Theorem 3.3]{shubh-jaa22} proved that an element of $T(X)$ is unit-regular if and only if it is semi-balanced.

%%%%%%%%%%%%%%%%%%%%%%%%%%%%
%%%%%%%%%%%%%%%%%%%%%%%%%%%%
%%%%%%%%%%%%%%%%%%%%%%%%%%%%
%%%%%%%%%%%%%%%%%%%%%%%%%%%%
\vspace{0.5mm}
The aim of this paper is to describe unit-regular and semi-balanced elements in $\overline{T}(X, Y)$, $\overline{L}(V, W)$, and $L(V)$. The rest of this paper is organized as follows. In the next section, we define concepts, introduce notation, and list some necessary results. In Section $3$, we describe unit-regular elements in $\overline{T}(X, Y)$ and determine when $\overline{T}(X, Y)$ is unit-regular. In Section $4$, we describe unit-regular elements in $\overline{L}(V, W)$. Using this, we prove that $f\in L(V)$ is unit-regular if and only if $\nul(f) = \corank(f)$. We also alternatively deduce that $L(V)$ is unit-regular if and only if $V$ is finite-dimensional. We further determine when $\overline{L}(V, W)$ is unit-regular. In Section $5$, we give necessary and sufficient conditions for the semigroups $\overline{T}(X, Y)$, $\overline{L}(V, W)$, and $L(V)$ to be semi-balanced.

%%%%%%%%%%%%%%%%%%%%%%%%%%%%
%%%%%%%%%%%%%%%%%%%%%%%%%%%%
%%%%%%%%%%%%%%%%%%%%%%%%%%%%
%%%%%%%%%%%%%%%%%%%%%%%%%%%%
\section{Preliminaries and Notation}
The cardinality of a set $A$ is denoted by $|A|$. Given any sets $A$ and $B$, let $A\setminus B \colonequals \{x\in A\colon x\notin B\}$. The set of all positive integers (resp. rational numbers) is denoted by $\mathbb{N}$ (resp. $\mathbb{Q}$), and let $\aleph_{0}\colonequals |\mathbb{N}|$. A \emph{transversal} of an equivalence relation $\rho$ on a nonempty set $X$ is a subset of $X$ that contains exactly one element from each $\rho$-class. Given a semigroup $S$ with identity, the set of all units of $S$ is denoted by $U(S)$; the set of all unit-regular elements in $S$ is denoted by $\ureg(S)$. 

%%%%%%%%%%%%%%%%%%%%%%%%%%%%
%%%%%%%%%%%%%%%%%%%%%%%%%%%%
%%%%%%%%%%%%%%%%%%%%%%%%%%%%
%%%%%%%%%%%%%%%%%%%%%%%%%%%%
%Functions-02
\vspace{0.5mm}
We compose maps from left to right and denote their composition by juxtaposition. Let $f\colon X \to Y$ be a map. The domain, codomain, and range of $f$ are denoted by $\dom(f)$, $\codom(f)$, and $R(\alpha)$, respectively. We use $xf$ to denote the image of an element $x$ of $X$ under $f$. 
Given any $A\subseteq \dom(f)$ (resp. $B\subseteq \codom(f)$), let $Af$ (resp. $Bf^{-1}$) denote the set $\{af\colon a\in A\}$ (resp. $\{x\in  \dom(f)\colon  xf \in B\}$). Moreover, we write $bf^{-1}$ instead of $\{b\} f^{-1}$ if $B=\{b\} \subseteq \codom(f)$.

%%%%%%%%%%%%%%%%%%%%%%%%%%%%
%%%%%%%%%%%%%%%%%%%%%%%%%%%%
%%%%%%%%%%%%%%%%%%%%%%%%%%%%
%%%%%%%%%%%%%%%%%%%%%%%%%%%%
\vspace{0.5mm}
Let $f\colon X \to X$ be a map and $A$ be a nonempty subset of $\dom(f)$. We say that $A$ is \emph{invariant} under $f$ if $Af \subseteq A$.
The \emph{restriction} of $f$ to $A$ is the map $f_{|_A} \colon A \to X$ defined by $a(f_{|_A}) = af$ for all $a \in A$. If $B \subseteq \codom(f)$ such that $Xf \subseteq B$, then the \emph{corestriction} of $f$ to $B$ is the map from $X$ to $B$ that agrees with $f$. Moreover, we use $f_{\upharpoonright_B}$ to denote the corestriction  of the map $f_{|_B}$ to $B$ if $B$ is invariant under $f$.
Note that if $f\in U(\overline{T}(X,Y))$, then $f_{\upharpoonright_Y}\in U(T(Y))$.

%%%%%%%%%%%%%%%%%%%%%%%%%%%%
%%%%%%%%%%%%%%%%%%%%%%%%%%%%
%%%%%%%%%%%%%%%%%%%%%%%%%%%%
%%%%%%%%%%%%%%%%%%%%%%%%%%%%
%Functions-02
\vspace{0.5mm}

Let $f\colon X \to X$ be a map. Set $D(f) \colonequals X\setminus R(f)$. The \emph{defect} $\textnormal{d}(f)$ of $f$ is the cardinality of $D(f)$. The \emph{kernel} of $f$, denoted by $\ker(f)$, is an equivalence relation on $X$ defined by $\ker(f) \colonequals \{(x,y)\in X \times X \colon xf = yf\}$. The symbol $\pi(f)$ denotes the partition of $X$ induced by $\ker(f)$, and $T_f$ denotes any transversal of the equivalence relation $\ker(f)$. Note that $|X\setminus T_f|$ is independent of the choice of transversal of $\ker(f)$ (cf. \cite[p. 1356]{hig-how-rus98}), and $|T_f| = |R(f)|$. The \emph{collapse} $\textnormal{c}(f)$ of $f$ is the cardinality of $X\setminus T_f$ (cf. \cite[p. 1356]{hig-how-rus98}). It is clear that $\textnormal{c}(f) = 0$ if and only if $f$ is injective. If $\textnormal{c}(f) = \textnormal{d}(f)$, then we say that $f$ is \emph{semi-balanced} (cf. \cite[p. 1356]{hig-how-rus98}). Let $B(S)$ denote the set of all semi-balanced elements in a subsemigroup $S$ of $T(X)$. Notice that $B(S)=S\cap B$.

%%%%%%%%%%%%%%%%%%%%%%%%%%%%
%%%%%%%%%%%%%%%%%%%%%%%%%%%%
%%%%%%%%%%%%%%%%%%%%%%%%%%%%
%%%%%%%%%%%%%%%%%%%%%%%%%%%%

\vspace{0.5mm}
Throughout this paper, unless otherwise noted, let $V$ denote a vector space over an arbitrary field. We denote by $0$ the zero vector of $V$. The subspaces $\{0\}$ and $V$ of $V$ are called \emph{trivial} subspaces. The subspace spanned by a subset $T$ of $V$ is denoted by $\langle T \rangle$. Let $U$ be a subspace of $V$. We say that $U$ is a \emph{proper} subspace of $V$ if $U \neq V$. The dimensions of $U$ and the quotient space $V/U$ are denoted by $\dim(U)$ and $\codim_V(U)$, respectively. We denote by $V = U \oplus W$ the (internal) direct sum of subspaces $U$ and $W$ of $V$. If $V = U \oplus W$, then we say that $W$ is a \emph{complement} of $U$ in $V$.

%%%%%%%%%%%%%%%%%%%%%%%%%%%%
%%%%%%%%%%%%%%%%%%%%%%%%%%%%
%%%%%%%%%%%%%%%%%%%%%%%%%%%%
%%%%%%%%%%%%%%%%%%%%%%%%%%%%

\vspace{0.5mm}
Let $f\in L(V)$. Set $R(f) \colonequals \{vf\colon v\in V\}$ and $N(f) \colonequals \{v\in V\colon vf = 0\}$. Note that $N(f)$ is a subspace of $\dom(f)$, and $R(f)$ is a subspace of $\codom(f)$.
We write $\nul(f)$, $\rank(f)$, and $\corank(f)$ to denote $\dim(N(f))$, $\dim(R(f))$, and $\dim(V/R(f))$, respectively. If $V$ and $U$ are isomorphic vector spaces, then we write $V \approx U$. Let $B$ and $B'$ be bases for $V$. We write $\bar{g}$ to denote the unique linear transformation on $V$ obtained by linear extension of a map $g\colon B \to V$ or a map $g\colon B \to B'$ to the entire space $V$ (cf. \cite[Theorem 2.2]{s-roman07}). Note that if $g\colon B \to B'$ is a bijective map, then $\bar{g}\in U(L(V))$. If $f\in \overline{L}(V,W)$, then $f_{\upharpoonright_W}\in L(W)$ (cf. \cite[p. 73]{s-roman07}).

%%%%%%%%%%%%%%%%%%%%%%%%%%%%
%%%%%%%%%%%%%%%%%%%%%%%%%%%%
%%%%%%%%%%%%%%%%%%%%%%%%%%%%
%%%%%%%%%%%%%%%%%%%%%%%%%%%%

\vspace{0.5mm}
We refer the reader to \cite{howie95} and \cite{s-roman07} for undefined terminology in semigroup theory and linear algebra, respectively. In the rest of the paper, $Y$ is a nonempty subset of a set $X$, and $W$ is a subspace of a vector space $V$ over a field. We end this section by stating a list of results needed in the sequel.

\begin{proposition}\cite[Proposition 1]{alar-s80}\label{alarsf80-unit-facto}
	A monoid is unit-regular if and only if it is factorizable.
\end{proposition}

\begin{proposition}\cite[Proposition 5]{alar-s80}\label{alar-sf80}
	The semigroup $T(X)$ is unit-regular if and only if $X$ is finite.
\end{proposition}

\begin{theorem}\cite[Theorem 2]{jamp-sa01}\label{facto-fin-dim}
	The semigroup $L(V)$ is factorizable if and only if $V$ is finite-dimensional.
\end{theorem}

\begin{lemma}\cite[Lemma 3.1]{shubh-jaa22}\label{jaa22-range-trans}
	Let $f\colon A \to B$ and $g\colon B \to A$ be maps. If $fgf = f$, then $R(fg)$ is a transversal of the equivalence relation $\ker(f)$ on $A$. 
\end{lemma}

\begin{theorem}\cite[Theorem 3.3]{shubh-jaa22} \label{ureg-smbal-equi}
	Let $f\in T(X)$. Then $f$ is unit-regular in $T(X)$ if and only if $f$ is semi-balanced.
\end{theorem}

\begin{lemma}\cite[Lemma 3.5]{shubh-jaa22} \label{transform-semibal-Xfinite}
	Every element of $T(X)$ is semi-balanced if and only if $X$ is finite.
\end{lemma}

%%%%%%%%%%%%%%%%%%%%%%%%%%%%
%%%%%%%%%%%%%%%%%%%%%%%%%%%%
%%%%%%%%%%%%%%%%%%%%%%%%%%%%
%%%%%%%%%%%%%%%%%%%%%%%%%%%%
%\newpage
\section{Unit-regular elements in $\overline{T}(X, Y)$}

In this section, we describe unit-regular elements in $\overline{T}(X, Y)$. Using this, we determine when $\overline{T}(X, Y)$ is unit-regular. We begin with the following lemma.

\begin{lemma}\label{inv-map-unit}
	Let $f\in \overline{T}(X, Y)$. If $f\in \textnormal{ureg}(\overline{T}(X, Y))$, then $f_{\upharpoonright_Y} \in \textnormal{ureg}(T(Y))$.
\end{lemma}

\begin{proof}[\textbf{Proof}]
	If $f\in \ureg(\overline{T}(X, Y))$, then there exists $g\in U(\overline{T}(X, Y))$ such that $fgf = f$. Therefore we get $f_{\upharpoonright_Y} g_{\upharpoonright_Y} f_{\upharpoonright_Y} = f_{\upharpoonright_Y}$ and $g_{\upharpoonright_Y}\in U(T(Y))$. Hence $f_{\upharpoonright_Y} \in \ureg(T(Y))$.
\end{proof}

Note that $\overline{T}(X, Y)$ contains the identity map on $X$. The next theorem describes unit-regular elements in $\overline{T}(X, Y)$.
\begin{theorem} \label{unit-reg-el-TXY}
	Let $f\in \overline{T}(X, Y)$. Then $f\in \textnormal{ureg}(\overline{T}(X, Y))$ if and only if
	\begin{enumerate}
		\item[\rm(i)] $f_{\upharpoonright_Y} \in \textnormal{ureg}(T(Y))$;
		\item[\rm(ii)] $R(f_{\upharpoonright_Y}) = Y \cap R(f)$;
		\item[\rm(iii)] $|C(f)\setminus C(f_{\upharpoonright_Y})|=|D(f)\setminus D(f_{\upharpoonright_Y})|$, where $C(f) \colonequals X \setminus T_f$ and $C(f_{\upharpoonright_Y}) \colonequals Y \setminus T_{(f_{\upharpoonright_Y})}$ for some transversals $T_f$ and $T_{(f_{\upharpoonright_Y})}$ of $\ker(f)$ and $\ker(f_{\upharpoonright_Y})$, respectively, such that $T_{(f_{\upharpoonright_Y})} = Y\cap T_f$.
	\end{enumerate}
\end{theorem}

\begin{proof}[\textbf{Proof}]
	Suppose that $f\in \ureg(\overline{T}(X, Y))$. Then $fgf = f$ for some $g\in U(\overline{T}(X, Y))$.	
	
	\begin{enumerate}
		\item[\rm(i)] By Lemma \ref{inv-map-unit}, we get $f_{\upharpoonright_Y} \in \ureg(T(Y))$. 
		
		\vspace{0.5mm}
		\item[\rm(ii)] It is clear that $R(f_{\upharpoonright_Y}) \subseteq Y \cap R(f)$, since $Yf\subseteq Y$. For the reverse inclusion, let $y\in Y \cap R(f)$. Then $xf = y$ for some $x\in X$. Therefore, since $f = fgf$, $y\in Y$, and $Yg \subseteq Y$, we obtain
		\[y = xf = (xf)gf = (yg)f\in R(f_{\upharpoonright_Y}).\] This implies $Y \cap R(f) \subseteq R(f_{\upharpoonright_Y})$ as required.

		\vspace{0.5mm}
		\item[\rm(iii)] Recall first that $fgf = f$. Then $f_{\upharpoonright_Y} g_{\upharpoonright_Y} f_{\upharpoonright_Y} = f_{\upharpoonright_Y}$. Therefore  by Lemma \ref{jaa22-range-trans}, the sets $T_f\colonequals R(fg)$ and $T_{(f_{\upharpoonright_Y})} \colonequals R(f_{\upharpoonright_Y} g_{\upharpoonright_Y})$ are transversals of $\ker(f)$ and $ \ker(f_{\upharpoonright_Y})$, respectively. Notice that $T_{(f_{\upharpoonright_Y})}=Y\cap T_f$. Set $C(f) \colonequals X \setminus T_f$ and $C(f_{\upharpoonright_Y}) \colonequals Y \setminus T_{(f_{\upharpoonright_Y})}$. Since $g$ and $g_{\upharpoonright_Y}$ are bijections, we obtain \[D(f)g = (X\setminus R(f))g = Xg \setminus R(fg) = X \setminus T_f = C(f) \quad \text{and}\] \[D(f_{\upharpoonright_Y})g_{\upharpoonright_Y} = (Y\setminus R(f_{\upharpoonright_Y}))g_{\upharpoonright_Y} = Yg_{\upharpoonright_Y} \setminus R(f_{\upharpoonright_Y} g_{\upharpoonright_Y}) = Y \setminus T_{(f_{\upharpoonright_Y})} = C(f_{\upharpoonright_Y}).\]
		Therefore  \[(D(f)\setminus D(f_{\upharpoonright_Y}))g=
		D(f)g\setminus D(f_{\upharpoonright_Y})g= D(f)g\setminus D(f_{\upharpoonright_Y})g_{\upharpoonright_Y}= C(f)\setminus C(f_{\upharpoonright_Y}).\]
		This implies $|C(f)\setminus C(f_{\upharpoonright_Y})|=|D(f)\setminus D(f_{\upharpoonright_Y})|$ as required.
	\end{enumerate}

	%\vspace{0.5mm}
	Conversely, suppose that $f$ satisfies the given conditions. To prove $f\in \textnormal{ureg}(\overline{T}(X, Y))$, we will construct $g\in U(\overline{T}(X, Y))$ such that $fgf = f$. For this, we first note from (iii) that $T_f \subseteq X$ and $T_{(f_{\upharpoonright_Y})} \subseteq Y$ are transversals of $\ker(f)$ and $\ker(f_{\upharpoonright_Y})$, respectively. Therefore the corestrictions of the maps $f_{|_{T_f}}$ and ${(f_{\upharpoonright_Y})}_{|_{T_{(f_{\upharpoonright_Y})}}}$ to $R(f)$ and $R(f_{\upharpoonright_Y})$, respectively, are bijections. Let $g_0$ be the inverse of the corestriction map of $f_{|_{T_f}}$ to $R(f)$, and let $h_0$ be the inverse of the corestriction map of ${(f_{\upharpoonright_Y})}_{|_{T_{(f_{\upharpoonright_Y})}}}$ to $R(f_{\upharpoonright_Y})$. It is routine to verify that $xg_0 = xh_0$ for all $x\in R(f_{\upharpoonright_Y})$, since $T_{(f_{\upharpoonright_Y})}=Y\cap T_f$ by \rm(iii). Therefore $R(f_{\upharpoonright_Y})g_0=T_{(f_{\upharpoonright_Y})}$. Set $C(f) \colonequals X\setminus T_f$ and $C(f_{\upharpoonright_Y})\colonequals Y\setminus T_{(f_{\upharpoonright_Y})}$. Combining \rm(i) and Theorem  \ref{ureg-smbal-equi}, we see that $f_{\upharpoonright_Y}$ is semi-balanced. Therefore there exists a bijection $g_1\colon D(f_{\upharpoonright_Y})\to C(f_{\upharpoonright_Y})$. Also by \rm(iii), there exists a bijection $g_2\colon D(f)\setminus D(f_{\upharpoonright_Y})\to C(f)\setminus C(f_{\upharpoonright_Y})$. Using these three bijections $g_0, g_1$, and $g_2$, define a map $g\colon X\to X$ by
	\begin{align*}
		xg=
		\begin{cases}
			xg_0 & \text{if $x\in R(f)$}\\
			xg_1 & \text{if $x\in D(f_{\upharpoonright_Y})$}\\
			xg_2 & \text{if $x\in D(f)\setminus D(f_{\upharpoonright_Y})$}.
		\end{cases}
	\end{align*}
	It is easy to check that $g$ is bijective. Moreover, we obtain
	\[Yg=(R(f_{\upharpoonright_Y})\cup D(f_{\upharpoonright_Y}))g =R(f_{\upharpoonright_Y})g_0\cup D(f_{\upharpoonright_ Y})g_1=T_{(f_{\upharpoonright_Y})}\cup C(f_{\upharpoonright_Y})\subseteq Y.\]
	Therefore $g\in U(\overline{T}(X, Y))$. We can also verify in a routine manner that $fgf=f$. Hence $f\in \ureg(\overline{T}(X, Y))$.
\end{proof}

%%%%%%%%%%%%%%%%%%%%%%%%%%%%
%%%%%%%%%%%%%%%%%%%%%%%%%%%%
%%%%%%%%%%%%%%%%%%%%%%%%%%%%
%%%%%%%%%%%%%%%%%%%%%%%%%%%%

To describe the unit-regularity of $\overline{T}(X, Y)$, we need the following lemma.

\begin{lemma}\label{Y-singletn-whole}
	We have $R(f_{\upharpoonright_Y})=Y\cap R(f)$ for all $f\in \overline{T}(X,Y)$ if and only if $|Y|=1$ or $Y=X$.
\end{lemma}

\begin{proof}[\textbf{Proof}]
	Suppose that $R(f_{\upharpoonright_Y})=Y\cap R(f)$ for all $f\in \overline{T}(X,Y)$. For $|X|\le 2$, the statement is trivial. Let us now assume that $|X|\ge 3$. Suppose to the contrary that $Y$ is a proper subset of $X$ such that $|Y|\ge 2$.  Let $a$ and $b$ be distinct elements of $Y$. Define a map $f\colon X\to X$ by
	\begin{align*}
		xf =
		\begin{cases}
			a & \text{if $x\in Y$}\\
			b & \text{otherwise}.
		\end{cases}
	\end{align*}
	It is clear that $Y$ is invariant under $f$, since $a\in Y$. Therefore $f\in \overline{T}(X,Y)$. However, we see that $R(f_{\upharpoonright_Y})=Yf=\{a\}$ while $Y\cap R(f) = \{a,b\}$, a contradiction. Hence $|Y|=1$ or $Y=X$.
	
	\vspace{0.5mm}	
	The proof of the converse part is immediate.
\end{proof}

%%%%%%%%%%%%%%%%%%%%%%%%%%%%
%%%%%%%%%%%%%%%%%%%%%%%%%%%%
%%%%%%%%%%%%%%%%%%%%%%%%%%%%
%%%%%%%%%%%%%%%%%%%%%%%%%%%%

The following theorem describes the unit-regularity of $\overline{T}(X, Y)$.

\begin{theorem}\label{unit-reg-TXY}
	The semigroup $\overline{T}(X,Y)$ is unit-regular if and only if
	\begin{enumerate}
		\item[\rm(i)] $|Y|=1$ or $Y=X$;
		\item[\rm(ii)] $X$ is finite.
	\end{enumerate}
\end{theorem}

\begin{proof}
	Suppose that $\overline{T}(X,Y)$ is unit-regular.
	\begin{enumerate}
		\item[\rm(i)] Then every element of $\overline{T}(X,Y)$ is unit-regular. Therefore by Theorem \ref{unit-reg-el-TXY}, we have $R(f_{\upharpoonright_Y})=Y\cap R(f)$ for all $f\in \overline{T}(X,Y)$. Hence $|Y|=1$ or $Y=X$ by Lemma \ref{Y-singletn-whole}.	
		
		\vspace{0.5mm}
		\item[\rm(ii)] Recall from \rm(i) that $Y=X$ or $|Y|=1$. If $Y = X$, then $\overline{T}(X,Y)=T(X)$. Therefore $X$ is finite by Proposition \ref{alar-sf80}. Let us now assume that $|Y|=1$. In this case, it suffices to prove that $X\setminus Y$ is finite. Suppose to the contrary that $X\setminus Y$ is infinite. Then there exists a map $\alpha \colon X\setminus Y\to X\setminus Y$ that is injective but not surjective. Define a map $f\colon X \to X$ by
		\begin{eqnarray*}
			xf=
			\begin{cases}
				x        & \text{if $x\in Y$}\\
				x\alpha  &  \text{otherwise.}
			\end{cases}
		\end{eqnarray*}
		Clearly $f\in \overline{T}(X,Y)$. Moreover, since $\alpha$ is injective but not surjective, the map $f$ is injective but not surjective. This gives $\textnormal{c}(f)=0$ and $\textnormal{d}(f)\ge 1$. Therefore $f$ is not semi-balanced, so $f\notin \ureg(T(X))$ by Theorem \ref{ureg-smbal-equi}. This implies $f\notin  \ureg(\overline{T}(X,Y))$, a contradiction. Hence $X\setminus Y$ is finite, as required.
	\end{enumerate}
	
	Conversely, suppose that the given conditions hold, and let $f\in \overline{T}(X,Y)$. In view of Theorem \ref{unit-reg-el-TXY}, in order to show that $f \in \ureg(\overline{T}(X,Y))$, it suffices to prove that $f$ satisfies all the conditions in Theorem \ref{unit-reg-el-TXY}. First, we note from (ii) that $Y$ is finite. Therefore $f_{\upharpoonright_Y} \in \ureg(T(Y))$ by Proposition \ref{alar-sf80}, so Theorem \ref{unit-reg-el-TXY}\rm(i) holds. Next, we recall (i). Then $R(f_{\upharpoonright_Y})=Y\cap R(f)$ by Lemma \ref{Y-singletn-whole}, so Theorem \ref{unit-reg-el-TXY}\rm(ii) holds.
	Finally, we note from (ii) that both $X$ and $Y$ are finite. Therefore $\mbox{c}(f)=\mbox{d}(f)$ and $\mbox{c}(f_{\upharpoonright_Y})=\mbox{d}(f_{\upharpoonright_Y})$ by Lemma \ref{transform-semibal-Xfinite}. We further notice that there exist transversals $T_f$ and $T_{(f_{\upharpoonright_Y})}$ of $\ker(f)$ and $\ker(f_{\upharpoonright_Y})$, respectively, such that $T_{(f_{\upharpoonright_Y})}=Y\cap T_f$. Set $C(f)\colonequals X\setminus T_f$ and $C(f_{\upharpoonright_Y})\colonequals Y\setminus T_{(f_{\upharpoonright_Y})}$. Then $|C(f)|= \textnormal{c}(f)= \textnormal{d}(f) = |D(f)|$ and $|C(f_{\upharpoonright_Y})|= \mbox{c}(f_{\upharpoonright_Y})= \mbox{d}(f_{\upharpoonright_Y})=|D(f_{\upharpoonright_Y})|$. Therefore, since $X$ is finite by \rm(i), we obtain
	\[|C(f)\setminus C(f_{\upharpoonright_Y})|=|C(f)|-|C(f_{\upharpoonright_Y})|=|D(f)|-|D(f_{\upharpoonright_Y})|=|D(f)\setminus D(f_{\upharpoonright_Y})|,\]
	so Theorem \ref{unit-reg-el-TXY}\rm(iii) holds. Thus $f \in \ureg(\overline{T}(X,Y))$ by Theorem \ref{unit-reg-el-TXY}. Hence, since $f\in \overline{T}(X,Y)$ is arbitrary, we conclude that $\overline{T}(X,Y)$ is unit-regular.
\end{proof}

%%%%%%%%%%%%%%%%%%%%%%%%%%%%
%%%%%%%%%%%%%%%%%%%%%%%%%%%%
%%%%%%%%%%%%%%%%%%%%%%%%%%%%
%%%%%%%%%%%%%%%%%%%%%%%%%%%%
%\newpage
\section{Unit-regular elements in $\overline{L}(V, W)$}

In this section, we describe unit-regular elements in $\overline{L}(V,W)$. Using this, we prove that an element $f$ of $L(V)$ is unit regular if and only if $\nul(f) = \corank(f)$. Moreover, we alternatively prove that $L(V)$ is unit-regular if and only if $V$ is finite-dimensional. We also give a new proof of \cite[Theorem 11]{chaiya-s19}, which determine the unit regularity of $\overline{L}(V,W)$. We first prove a list of lemmas, which assist to describe unit-regular elements in $\overline{L}(V,W)$.

%%%%%%%%%%%%%%%%%%%%%%%%%%%%
%%%%%%%%%%%%%%%%%%%%%%%%%%%%
%%%%%%%%%%%%%%%%%%%%%%%%%%%%
%%%%%%%%%%%%%%%%%%%%%%%%%%%%

\begin{lemma}\label{kerf-iso-kergf}
	Let $f, g \in L(V)$. If $g$ is injective and $N(f)\subseteq R(g)$, then $N(gf) \approx N(f)$.
\end{lemma}

\begin{proof}[\textbf{Proof}]
	First, we prove that $vg \in N(f)$ whenever $v\in N(gf)$. For this, let $v\in N(gf)$. Then $v(gf) = 0$. Therefore $(vg)f = v(gf) = 0$, so $vg \in N(f)$.
	
	\vspace{0.5mm}
	Now, define a map $\varphi \colon N(gf) \to N(f)$ by $v\varphi = vg$ for all $v\in N(gf)$. It is clear that $\varphi$ is also  a monomorphism, since $g$ is a monomorphism. To prove $\varphi$ is surjective, let $w\in N(f)$. Then, since $N(f)\subseteq R(g)$, we have $w\in R(g)$. Since $g$ is injective, we see that the corestriction of $g$ to $R(g)$ is bijective. Therefore there exists a unique $u\in V$ such that $wg^{-1} = u$. It remains to show that $u \in N(gf)$ and $u\varphi = w$. These two conditions can easily verify in a routine manner. Hence $\varphi$ is surjective. Thus $N(gf) \approx N(f)$.
\end{proof}

%%%%%%%%%%%%%%%%%%%%%%%%%%%%
%%%%%%%%%%%%%%%%%%%%%%%%%%%%
%%%%%%%%%%%%%%%%%%%%%%%%%%%%
%%%%%%%%%%%%%%%%%%%%%%%%%%%%

Let $U$ and $V$ be isomorphic vector spaces. If $U'$ is a subspace of $U$ and $V'$ is a subspace of $V$ such that $U' \approx V'$, then it is not necessarily true that $U/U' \approx V/V'$ (cf. \cite[p. 93, line 7]{s-roman07}). However, we have the following lemma.

\begin{lemma}\label{spa-iso-quo-iso}
	Let $U'$ and $V'$ be subspaces of vector spaces $U$ and $V$, respectively. If there exists an isomorphism $f\colon U\to V$ such that $U'f=V'$, then $U/U'\approx V/V'$.
\end{lemma}

\begin{proof}[\textbf{Proof}]
	Define a map $\varphi \colon U\to V/V'$ by $u\varphi =uf+V'$. By using linearity of $f$, we see that $\varphi$ is linear. To prove $\varphi$ is surjective, let $v+V'\in V/V'$. Then, since $v\in V$ and $f\colon U \to V$ is bijective, there exists a unique $u\in U$ such that $uf = v$. Therefore $u\varphi =uf+V'=v+V'$, so $\varphi$ is surjective. Thus $U/N(\varphi)\approx V/V'$ by the first isomorphism theorem.
	
	\vspace{0.5mm}
	Finally, we prove that $N(\varphi)=U'$. Let $u\in N(\varphi)$. Then $uf+V' = u\varphi =V'$, so $uf\in V'$. Therefore, since $f$ is bijective and $U'f=V'$, we get $u\in U'$. Thus $N(\varphi)\subseteq U'$. For the reverse inclusion, let $u\in U'$. Then, since $U'f=V'$, we have $uf\in V'$. Therefore $u\varphi=uf+V'=V'$, so $u\in N(\varphi)$. Thus $U'\subseteq N(\varphi)$. Hence $N(\varphi)=U'$ and consequently $U/U'\approx V/V'$.
\end{proof}

%%%%%%%%%%%%%%%%%%%%%%%%%%%%%%%
%%%%%%%%%%%%%%%%%%%%%%%%%%%%%%%
%%%%%%%%%%%%%%%%%%%%%%%%%%%%%%%
%%%%%%%%%%%%%%%%%%%%%%%%%%%%%%%

\begin{lemma}\label{subspace-transversal}
	Let $f\in \overline{L}(V,W)$. Then there exists a subspace $U$ of $V$ such that $U$ and $U \cap W$ are transversals of $\ker(f)$ and $\ker(f_{\upharpoonright_W})$, respectively.
	
	%Let $f\in L(V)$ be such that $Wf \subseteq W$. Then there exists a subspece $U$ of $V$ such that $U$ and $U \cap W$ are transversals of $\ker(f)$ and $\ker(f_{\upharpoonright_W})$, respectively.
\end{lemma}

\begin{proof}[\textbf{Proof}]
	Let $B_1$ be a basis for $R(f_{\upharpoonright_W})$. If $B_1\neq \varnothing$, then we have $uf^{-1}\cap W\neq \varnothing$ for all $u\in B_1$. Therefore  fix $u'\in uf^{-1}\cap W$ for each $u\in B_1$, and let $C_1\colonequals \{u'\colon u\in B_1\}$. For $B_1=\varnothing$, we let $C_1\colonequals\varnothing$.
	
	\vspace{0.5mm}
	In either case, we prove that $U_1\colonequals \langle C_1\rangle$ is a transversal of $\ker(f_{\upharpoonright_W})$. If $C_1=\varnothing$, then $U_1=\{0\}$, and we are done. Assume that $C_1\neq\varnothing$, and let $w\in R(f_{\upharpoonright_W})$. Then, since $B_1$ is a basis for $R(f_{\upharpoonright_W})$, we have $w=c_1u_1+\cdots +c_mu_m$ for some $u_1,\ldots ,u_m\in B_1$, where $m\geq 0$. Let $w'= c_1u_1'+\cdots +c_mu_m' \in U_1$, where $u_1',\ldots, u_m'\in C_1$. Notice that $w'f=w$, so $w'\in wf^{-1}\cap U_1$. Therefore, by construction of basis $C_1$ for $U_1$, we see that $wf^{-1}\cap U_1=\{w'\}$ and consequently $|wf^{-1}\cap U_1|=1$. Hence, since $w$ is arbitrary, we conclude that the subspace $U_1$ of $V$ is a transversal of $\ker{(f_{\upharpoonright_W})}$. 
	
	\vspace{0.5mm}
	Now, consider $R(f)\setminus R(f_{\upharpoonright_W})$. Then there are two possibilities to consider.
	
	\vspace{0.5mm}
	\noindent Case 1: Suppose $R(f)\setminus R(f_{\upharpoonright_W})=\varnothing$. Set $U\colonequals U_1$. Then $U\cap W=U_1$. Therefore $U$ and $U\cap W$ are transversals of $\ker(f)$ and $\ker(f_{\upharpoonright_W})$, respectively.
	
	\vspace{0.5mm}
	\noindent Case 2: Suppose $R(f)\setminus R(f_{\upharpoonright_W})\neq\varnothing$. Recall that $B_1$ is a basis for $R(f_{\upharpoonright_W})$. Extend $B_1$ to a basis $B_1\cup B_2$ for $R(f)$, where $B_2\subseteq R(f)\setminus R(f_{\upharpoonright_W})$. Clearly $B_2\neq \varnothing$. Note that $vf^{-1}\neq \varnothing$ for all $v\in B_2$. Therefore fix $\bar{v}\in vf^{-1}$ for each $v\in B_2$. Set $C\colonequals C_1\cup \{\bar{v} \in vf^{-1}\colon v\in B_2\}$ and $U\colonequals \langle C\rangle$. It is easy to see that $U\cap W=U_1$. Therefore $U\cap W$ is a transversal of $\ker(f_{\upharpoonright_W})$. It remains to prove that $U$ is a transversal of $\ker(f)$. For this, let $w\in R(f)$. Then, since $B_1\cup B_2$ is a basis for $R(f)$, we have $w=c_1u_1+\cdots +c_mu_m+d_1v_1+\cdots +d_nv_n$ for some $u_1,\ldots ,u_m\in B_1$ and $v_1,\ldots ,v_n\in B_2$, where $m,n\geq 0$. Let $w'= c_1u_1'+\cdots +c_mu_m'+d_1\bar{v_1}+\cdots +d_n\bar{v_n} \in U$, where $u_1',\ldots, u_m',\bar{v_1},\ldots,\bar{v_n}\in C$. Notice that $w'f=w$, so $w'\in wf^{-1}\cap U$. Therefore, by construction of basis $C$ for $U$, we get $wf^{-1}\cap U=\{w'\}$ and consequently $|wf^{-1}\cap U|=1$. Hence, since $w$ is arbitrary, the subspace $U$ of $V$ is a transversal of $\ker{(f)}$.
\end{proof}

%%%%%%%%%%%%%%%%%%%%%%%%%%%%
%%%%%%%%%%%%%%%%%%%%%%%%%%%%
%%%%%%%%%%%%%%%%%%%%%%%%%%%%
%%%%%%%%%%%%%%%%%%%%%%%%%%%%
\begin{lemma}\label{bases-n-r-tran}
	Let $f\in L(V)$ and $T_f$ be a transversal of $\ker(f)$. If $B_0$ is a basis for $N(f)$ and $B$ is a basis for $R(f)$, then $B_0\cup (T_f\cap Bf^{-1})$ is a basis for $V$.
\end{lemma}
\begin{proof}[\textbf{Proof}]
	We first note that $B_0\cap (T_f\cap Bf^{-1})=\varnothing$. It is easy to see that $T_f \cap Bf^{-1}$ is linearly independent. Combining the previous two statements shows that $B_0\cup (T_f\cap Bf^{-1})$ is linearly independent. It remains to prove that $\langle B_0\cup (T_f\cap Bf^{-1})\rangle = V$. For this, consider a complement $U$ of $N(f)$ in $V$. Then $B_0\cup B(f_{|_U})^{-1}$ is a basis for $V$, where the isomorphism $f_{|_U}$ also denotes the corestriction of $f_{|_U}$ to $R(f)$. Now, let $u\in B(f_{|_U})^{-1}$. Clearly $u\in U$. Then $v \colonequals u+u'\in T_f\cap Bf^{-1}$ for some $u'\in N(f)$. Therefore $u=v-u'\in \langle B_0\cup (T_f\cap Bf^{-1})\rangle$. Hence, since $u$ is arbitrary, we have $B(f_{|_U})^{-1}\subseteq \langle B_0\cup (T_f\cap Bf^{-1})\rangle$ and consequently $B_0 \cup B(f_{|_U})^{-1}\subseteq \langle B_0\cup (T_f\cap Bf^{-1})\rangle$. Thus $\langle B_0\cup (T_f\cap Bf^{-1})\rangle = V$ as required.
\end{proof}

%%%%%%%%%%%%%%%%%%%%%%%%%%%%
%%%%%%%%%%%%%%%%%%%%%%%%%%%%
%%%%%%%%%%%%%%%%%%%%%%%%%%%%
%%%%%%%%%%%%%%%%%%%%%%%%%%%%

\begin{lemma}\label{lin-res-unit}
	Let $f\in \overline{L}(V,W)$. If $f\in \textnormal{ureg}(\overline{L}(V,W))$, then $f_{\upharpoonright_W}\in \textnormal{ureg}(L(W))$.
\end{lemma}

\begin{proof}[\textbf{Proof}]
	If $f\in \ureg(\overline{L}(V, W))$, then $fgf = f$ for some $g\in U(\overline{L}(V, W))$. Therefore we get  $f_{\upharpoonright_W} g_{\upharpoonright_W} f_{\upharpoonright_W} = f_{\upharpoonright_W}$ and $g_{\upharpoonright_W}\in U(L(W))$. Hence $f_{\upharpoonright_W} \in \ureg(L(W))$.
\end{proof}

%%%%%%%%%%%%%%%%%%%%%%%%%%%%
%%%%%%%%%%%%%%%%%%%%%%%%%%%%
%%%%%%%%%%%%%%%%%%%%%%%%%%%%
%%%%%%%%%%%%%%%%%%%%%%%%%%%%

Note that $\overline{L}(V, U)$ contains the identity map on $V$. The following theorem describes unit-regular elements in $\overline{L}(V,W)$.

\begin{theorem}\label{char-unit-reg-lin-map}
	Let $f\in \overline{L}(V,W)$. Then $f \in \textnormal{ureg}(\overline{L}(V,W))$ if and only if
	\begin{enumerate}
		\item[\rm(i)] $R(f_{\upharpoonright_W})=W\cap R(f)$;
		\item[\rm(ii)] $\textnormal{nullity}(f_{\upharpoonright_W}) = \textnormal{corank}(f_{\upharpoonright_W})$;
		\item[\rm(iii)] $\textnormal{codim}_V(W+T_f)= \textnormal{codim}_V(W+R(f))$ for some subspace $T_f$ of $V$ such that $T_f$ and $W\cap T_f$ are transversals of $\ker(f)$ and $\ker(f_{\upharpoonright_W})$, respectively.
	\end{enumerate}
\end{theorem}

\begin{proof}[\textbf{Proof}]
	Suppose that $f \in \ureg(\overline{L}(V,W))$. Then $fgf=f$ for some $g\in U(\overline{L}(V,W))$.
	
	\begin{enumerate}
		\item[\rm(i)] It is clear that $R(f_{\upharpoonright_W})\subseteq W\cap R(f)$, since $R(f_{\upharpoonright_W})\subseteq W$. For the reverse inclusion, let $w\in W\cap R(f)$. Then there exists $v\in V$ such that $vf = w$. Therefore, since $w \in W$, $f= fgf$, and
		$Wg \subseteq W$, we obtain
		\[w=vf=(vf)gf=(wg)f\in R(f_{\upharpoonright_W}).\]
		This implies $W\cap R(f) \subseteq R(f_{\upharpoonright_W})$ as required.

		\vspace{0.2mm}
		\item[\rm(ii)] By Lemma \ref{lin-res-unit}, we have $f_{\upharpoonright_W}\in \textnormal{ureg}(L(W))$. Therefore there exists $g_{\upharpoonright_W}\in U(L(W))$ such that $f_{\upharpoonright_W} g_{\upharpoonright_W} f_{\upharpoonright_W}=f_{\upharpoonright_W}$. Note that  $R(g_{\upharpoonright_W} f_{\upharpoonright_W})= R(f_{\upharpoonright_W}) $ and $g_{\upharpoonright_W} f_{\upharpoonright_W}$ is an idempotent of $L(W)$. Combining this with \cite[Theorem 2.22]{s-roman07}, we obtain $N(g_{\upharpoonright_W} f_{\upharpoonright_W})\oplus R(f_{\upharpoonright_W})=W$. Therefore $N(g_{\upharpoonright_W} f_{\upharpoonright_W}) \approx W/R(f_{\upharpoonright_W})$. Since $g_{\upharpoonright_W}$ is bijective, it follows from Lemma \ref{kerf-iso-kergf} that
		$N(f_{\upharpoonright_W}) \approx N(g_{\upharpoonright_W} f_{\upharpoonright_W})$ and subsequently $N(f_{\upharpoonright_W}) \approx W/R(f_{\upharpoonright_W})$. Hence
		$\nul(f_{\upharpoonright_W}) = \corank(f_{\upharpoonright_W})$.
		
		\vspace{0.2mm}
		\item[\rm(iii)] Note that $R(fg)$ and $R(f_{\upharpoonright_W} g_{\upharpoonright_W})$ are subspaces of $V$ and $W$, respectively. Let $T_f\colonequals R(fg)$ and $T_{(f_{\upharpoonright_W})}\colonequals R(f_{\upharpoonright_W} g_{\upharpoonright_W})$. Recall that $fgf=f$ and $f_{\upharpoonright_W} g_{\upharpoonright_W} f_{\upharpoonright_W}= f_{\upharpoonright_W}$. Therefore by Lemma \ref{jaa22-range-trans}, it follows that $T_f$ and $T_{(f_{\upharpoonright_W})}$ are transversals of $\ker(f)$ and $\ker(f_{\upharpoonright_W})$, respectively. Moreover, we establish that $T_{(f_{\upharpoonright_W})}=W\cap T_f$. Clearly $T_{(f_{\upharpoonright_W})}\subseteq W\cap T_f$. For the reverse inclusion, let $v\in W\cap T_f$. Then $u(fg)=v$ for some $u\in V$. Therefore, since $g$ is bijective, we get $uf\in W$. Hence by \rm(i), there exists $u'\in W$ such that $u'f=uf$. This implies $u'(fg)=u(fg)=v$, so $v\in T_{(f_{\upharpoonright_W})}$ as required. Also, since $g$ is bijective, we see that $(W+R(f))g=W+T_f$. Thus $V/(W+R(f))\approx V/(W+T_f)$ by Lemma \ref{spa-iso-quo-iso}. Hence $\codim_V(W+T_f)= \codim_V(W+R(f))$.
	\end{enumerate}

	Conversely, suppose that the given conditions hold for $f\in \overline{L}(V,W)$. To prove $f \in \ureg(\overline{L}(V,W))$, we will construct $g\in U(\overline{L}(V,W))$ such that $fgf = f$. For this, let $B_1$ be a basis for $R(f)\cap W$. Extend $B_1$ to bases $B_1 \cup B_2$ and $B_1 \cup B_3$ for $R(f)$ and $W$, respectively, where $B_2 \subseteq R(f)\setminus \langle B_1 \rangle$ and $B_3 \subseteq W\setminus \langle B_1 \rangle$. It is easy to see that $B_1\cup B_2 \cup B_3$ is a basis for $W+ R(f)$. Now, extend $B_1\cup B_2 \cup B_3$ to a basis $B \colonequals B_1\cup B_2 \cup B_3\cup B_4$ for $V$, where $B_4 \subseteq V\setminus (W+ R(f))$.
	
	\vspace{0.2mm}
	Note from (iii) that the subspace $T_f$ of $V$ is a transversal of $\ker(f)$. Therefore the corestriction of $f_{|_{T_f}}$ to $R(f)$ is an isomorphism. Denote by $g_0$ the inverse of this corestriction map. Clearly $g_0\colon R(f)\to T_f$ is an isomorphism. Set $B_1g_0\colonequals C_1$ and $B_2g_0\colonequals C_2$.  It is evident that $C_1\cup C_2$ is a basis for $T_f$, since $B_1 \cup B_2$ is a basis for $R(f)$. Next, note from \rm(iii) that the subspace $T_{(f_{\upharpoonright_W})}\colonequals W\cap T_f$ of $W$ is a transversal of $\ker(f_{\upharpoonright_W})$. Therefore the corestriction of ${(f_{\upharpoonright_W})}_{|_{T_{(f_{\upharpoonright_W})}}}$ to $R(f_{\upharpoonright_W})$ is an isomorphism. Notice that the inverse of this isomorphism agrees with $g_0$ on $R(f_{\upharpoonright_W})$. Recall that $B_1$ is a basis for $W\cap R(f)$. Therefore $C_1$ is a basis for $T_{(f_{\upharpoonright_W})}$, since $R(f_{\upharpoonright_W})=W\cap R(f)$ by \rm(i).
	
	\vspace{0.5mm}
	Let $C_3$ be a basis for $N(f_{\upharpoonright_W})$. Then $C_1\cup C_3$ is a basis for $W$ by Lemma \ref{bases-n-r-tran}, and subsequently $C_1\cup C_2\cup C_3$ is a basis for $W+T_f$. Extend $C_1\cup C_2\cup C_3$ to a basis $C\colonequals C_1\cup C_2\cup C_3\cup C_4$ for $V$, where $C_4\subseteq V\setminus (W+T_f)$. Recall that $B_1$ and $B_1\cup B_3$ are bases for
	$W\cap R(f)$ and $W$, respectively. Since $R(f_{\upharpoonright_W}) = W \cap R(f)$ by (i), it follows that $\corank(f_{\upharpoonright_W})=|B_3|$. Therefore $|C_3|= \nul(f_{\upharpoonright_W}) = |B_3|$ by (ii). Thus there exists a bijection $\alpha \colon B_3\to C_3$. Next, recall that $B$ and $C$ are bases for $V$ such that $B_1\cup B_2\cup B_3 \subseteq B$ and $C_1\cup C_2\cup C_3 \subseteq C$. Since $B_1\cup B_2\cup B_3$ is a basis for $W+R(f)$ and $C_1\cup C_2\cup C_3$ is basis for $W+T_f$, we get $\codim_V(W+R(f))=|B_4|$ and $\codim_V(W+T_f)=|C_4|$. Therefore $|B_4|=|C_4|$ by (iii). Thus there exists a bijection $\beta \colon B_4\to C_4$.
	
	\vspace{0.2mm}
	Now, define a map $g\colon B \to C$ by setting
	\begin{align*}
		vg=
		\begin{cases}
			vg_0 & \text{if $v\in B_1\cup B_2$}\\
			v\alpha & \text{if $v\in B_3$}\\
			v\beta & \text{if $v\in B_4$}.
		\end{cases}
	\end{align*}
	It is easy to check that $g$ is bijective. Therefore the unique linear map $\bar{g}\colon V \to V$ is bijective. Moreover, we obtain 
	$W\bar{g}=\langle B_1\cup B_3\rangle \bar{g}=\langle B_1g_0\cup B_3\alpha \rangle =\langle C_1\cup C_3\rangle\subseteq W$. Thus
	$\bar{g}\in U(\overline{L}(V,W))$. We can also verify in a routine manner that $f\bar{g}f=f$. Hence $f \in \ureg(\overline{L}(V,W))$.
\end{proof}

%%%%%%%%%%%%%%%%%%%%%%%%%%%%
%%%%%%%%%%%%%%%%%%%%%%%%%%%%
%%%%%%%%%%%%%%%%%%%%%%%%%%%%
%%%%%%%%%%%%%%%%%%%%%%%%%%%%

If $W = V$, then conditions (i) and (iii) of Theorem \ref{char-unit-reg-lin-map} trivially hold. The following is an immediate corollary of Theorem \ref{char-unit-reg-lin-map} for the case when $W = V$, which describes unit-regular elements in $L(V)$.

\begin{corollary}\label{nul-corank-unit-reg-ele}
	Let $f\in L(V)$. Then $f \in \textnormal{ureg}(L(V))$ if and only if $\textnormal{nullity}(f)= \textnormal{corank}(f)$.
\end{corollary}

%%%%%%%%%%%%%%%%%%%%%%%%%%%%
%%%%%%%%%%%%%%%%%%%%%%%%%%%%
%%%%%%%%%%%%%%%%%%%%%%%%%%%%
%%%%%%%%%%%%%%%%%%%%%%%%%%%%
%\vspace{0.1cm}
We need the following lemma to give an alternative proof of Theorem \ref{unit-reg-fin-dim}, which may also be obtained by combining  Proposition \ref{alarsf80-unit-facto} and Theorem \ref{facto-fin-dim}.

\begin{lemma}\label{nul-corank-fin-dim}
	We have	$\textnormal{nullity}(f)=\textnormal{corank}(f)$ for all $f\in L(V)$ if and only if $V$ is finite-dimensional.
\end{lemma}

%The sufficiency part of the problem has been studied
\begin{proof}[\textbf{Proof}]
	%($\Longrightarrow :$) We proceed by contraposition. 
	We prove the necessity part of the lemma by contraposition. Suppose that $V$ is infinite-dimensional. Let $v$ be a nonzero vector of $V$. Clearly $\dim(\langle v\rangle) = 1$. Note that $V=\langle v\rangle \oplus U$ for some complement $U$ of $\langle v\rangle$ in $V$. Then by \cite[Theorem 1.14]{s-roman07}, we get $\dim(V)=\dim(U)$. Therefore there is an isomorphism $\varphi \colon V\to U$. Since $U$ is a proper subspace of $V$, it follows that there is an injective linear map $\tilde{\varphi}\in L(V)$ that agrees with $\varphi$ on $V$. Thus $\nul(\tilde{\varphi})=0$. It is also immediate that $\corank(\tilde{\varphi})=1$. Hence $\textnormal{nullity}(\tilde{\varphi})\neq\textnormal{corank}(\tilde{\varphi})$.
	
	\vspace{0.2mm}
	Conversely, suppose that $V$ is finite-dimensional, and let $f\in L(V)$. Then $\corank(f)=\dim(V)- \rank(f)$. Therefore by rank-nullity theorem, we obtain $\nul(f)=\dim(V)-\rank(f)=\corank(f)$ as required.
\end{proof}

%%%%%%%%%%%%%%%%%%%%%%%%%%%%
%%%%%%%%%%%%%%%%%%%%%%%%%%%%
%%%%%%%%%%%%%%%%%%%%%%%%%%%%
%%%%%%%%%%%%%%%%%%%%%%%%%%%%

The following known theorem is straightforward from Corollary \ref{nul-corank-unit-reg-ele} and Lemma \ref{nul-corank-fin-dim}.

\begin{theorem}\label{unit-reg-fin-dim}
	The semigroup $L(V)$ is unit-regular if and only if $V$ is finite-dimensional.
\end{theorem}

%%%%%%%%%%%%%%%%%%%%%%%%%%%%
%%%%%%%%%%%%%%%%%%%%%%%%%%%%
%%%%%%%%%%%%%%%%%%%%%%%%%%%%
%%%%%%%%%%%%%%%%%%%%%%%%%%%%

Next, we need the following lemma to prove Theorem \ref{L-V-W=unit-reg}

\begin{lemma}\label{int-imp-triv}
	We have	$R(f_{\upharpoonright_W})=W\cap R(f)$ for all $f\in \overline{L}(V,W)$ if and only if $W$ is trivial.
\end{lemma}
\begin{proof}[\textbf{Proof}]
	We prove the necessity part of the lemma by contraposition. Suppose that $W$ is nontrivial. Then there exists a nontrivial subspace $U$ of $V$ such that $V=W\oplus U$. Let $B$ and $C$ be bases for $W$ and $U$, respectively. Then $B\cup C$ is a basis for $V$. Fix $w \in B$ and $u\in C$. Define a map $f\colon B \cup C \to V$ by
	\begin{equation*}
		vf=
		\begin{cases}
			w  & \text{if $v=u$}\\
			0 & \text{if $v\in B\cup (C\setminus \{u\})$}.
		\end{cases}
	\end{equation*}
	It is easy to see that $B\bar{f} = Bf=\{0\}$. Therefore $W\bar{f}=\{0\}\subseteq W$, so $\bar{f}\in \overline{L}(V,W)$. We further see that $W\cap R(\bar{f}) = \langle w\rangle$. Therefore $R(\bar{f}_{\upharpoonright_W})\neq W\cap R(\bar{f})$.
	
	\vspace{0.5mm}
	The proof of the converse part is immediate.
\end{proof}

%%%%%%%%%%%%%%%%%%%%%%%%%%%%
%%%%%%%%%%%%%%%%%%%%%%%%%%%%
%%%%%%%%%%%%%%%%%%%%%%%%%%%%
%%%%%%%%%%%%%%%%%%%%%%%%%%%%

We can now obtain a new proof of the following known theorem (cf. \cite[Theorem 11]{chaiya-s19}).
\begin{theorem}\label{L-V-W=unit-reg}
	The semigroup $\overline{L}(V,W)$ is unit-regular if and only if
	\begin{enumerate}
		\item[\rm(i)] $W$ is trivial;
		\item[\rm(ii)] $V$ is finite-dimensional.
	\end{enumerate}
\end{theorem}

\begin{proof}[\textbf{Proof}]
	Suppose that $\overline{L}(V,W)$ is unit-regular.
	\begin{enumerate}
		\item[\rm(i)] By Theorem \ref{char-unit-reg-lin-map}, we have $R(f_{\upharpoonright_W})=W\cap R(f)$ for all $f\in \overline{L}(V,W)$. Therefore $W$ is trivial by Lemma \ref{int-imp-triv}.
		
		\vspace{0.5mm}
		\item[\rm(ii)] From (i), it is clear that $\overline{L}(V,W) = L(V)$. Then by assumption, the semigroup $L(V)$ is unit-regular. Therefore $V$ is finite-dimensional by Theorem \ref{unit-reg-fin-dim}.
	\end{enumerate}
	
	Conversely, suppose that the given conditions hold. Then by (i), we have $\overline{L}(V,W) = L(V)$. Therefore by combining (ii) and Theorem \ref{facto-fin-dim}, the semigroup $\overline{L}(V,W)$ is factorizable. Hence $\overline{L}(V,W)$ is unit-regular by Proposition \ref{alarsf80-unit-facto}.
\end{proof}

%%%%%%%%%%%%%%%%%%%%%%%%%%%%
%%%%%%%%%%%%%%%%%%%%%%%%%%%%
%%%%%%%%%%%%%%%%%%%%%%%%%%%%
%%%%%%%%%%%%%%%%%%%%%%%%%%%%
%\newpage
\section{Semi-Balanced Semigroups}
Recall that an element $f$ of $T(X)$ is semi-balanced if $\textnormal{c}(f) =\textnormal{d}(f)$; a subsemigroup of $T(X)$ is semi-balanced if all its elements are semi-balanced. In this section, we give necessary and sufficient conditions for the semigroups $\overline{T}(X, Y)$,  $L(V)$, and $\overline{L}(V, W)$ to be semi-balanced. We begin with the following trivial remark.

\begin{remark}\label{subsemigrp=semi-balanced}
	Every subsemigroup of a semi-balanced semigroup is semi-balanced.
\end{remark}	

%%%%%%%%%%%%%%%%%%%%%%%%%%%%
%%%%%%%%%%%%%%%%%%%%%%%%%%%%
%%%%%%%%%%%%%%%%%%%%%%%%%%%%
%%%%%%%%%%%%%%%%%%%%%%%%%%%%

The following theorem is immediate from Proposition \ref{alar-sf80} and Lemma \ref{transform-semibal-Xfinite}. 

%\begin{theorem}\label{FTS-semibal-finite}
%The semigroup $T(X)$ is semi-balanced if and only if $X$ is finite.
%\end{theorem}

\begin{theorem}\label{FTS-semibal-finite} %\label{ureg-smbal-TX}
	The following statements are equivalent.
	\begin{enumerate}
		\item[\rm(i)] The semigroup $T(X)$ is unit-regular.
		\item[\rm(ii)] The semigroup $T(X)$ is semi-balanced.
		\item[\rm(iii)] The set $X$ is finite.	
	\end{enumerate}
\end{theorem}

%%%%%%%%%%%%%%%%%%%%%%%%%%%%
%%%%%%%%%%%%%%%%%%%%%%%%%%%%
%%%%%%%%%%%%%%%%%%%%%%%%%%%%
%%%%%%%%%%%%%%%%%%%%%%%%%%%%

Using Theorem \ref{ureg-smbal-equi}, we prove the next proposition.

\begin{proposition}\label{unit-semibal}
	If $M$ is a submonoid of $T(X)$, then $\textnormal{ureg}(M) \subseteq B(M)$.
\end{proposition}

\begin{proof}[\textbf{Proof}]
	Let $f\in \textnormal{ureg}(M)$. Then $f\in \textnormal{ureg}(T(X))$, since $M$ is a submonoid of $T(X)$. Therefore $f\in B$ by Theorem \ref{ureg-smbal-equi}. Hence $f\in M \cap B = B(M)$ as required.	
\end{proof}

%%%%%%%%%%%%%%%%%%%%%%%%%%%%
%%%%%%%%%%%%%%%%%%%%%%%%%%%%
%%%%%%%%%%%%%%%%%%%%%%%%%%%%
%%%%%%%%%%%%%%%%%%%%%%%%%%%%

%\vspace{0.5mm}
The following is an immediate corollary of Proposition \ref{unit-semibal}.

\begin{corollary}\label{unit-reg--semi-balanced}
	Every unit-regular submonoid of $T(X)$ is semi-balanced.
\end{corollary}

%%%%%%%%%%%%%%%%%%%%%%%%%%%%
%%%%%%%%%%%%%%%%%%%%%%%%%%%%
%%%%%%%%%%%%%%%%%%%%%%%%%%%%
%%%%%%%%%%%%%%%%%%%%%%%%%%%%
The converse of Corollary \ref{unit-reg--semi-balanced} is not necessarily true as seen from the following example.

\begin{example}
	Let $X$ be a finite set such that $|X|\ge 3$, and let $Y$ be a proper subset of $X$ such that $|Y|\ge 2$. It is clear that $\overline{T}(X, Y)$ is semi-balanced, since $X$ is finite. However, the semigroup $\overline{T}(X, Y)$ is not regular by \cite[Proposition 2.1(ii)]{nenth06}. Therefore $\overline{T}(X, Y)$ is not unit-regular.
\end{example}

%%%%%%%%%%%%%%%%%%%%%%%%%%%%
%%%%%%%%%%%%%%%%%%%%%%%%%%%%
%%%%%%%%%%%%%%%%%%%%%%%%%%%%
%%%%%%%%%%%%%%%%%%%%%%%%%%%%

In the next theorem, we give a necessary and sufficient condition for the semigroup $\overline{T}(X, Y)$ to be semi-balanced.

\begin{theorem}\label{T-XY-semibal-finite}
	The semigroup $\overline{T}(X,Y)$ is semi-balanced if and only if $X$ is finite.
\end{theorem}

\begin{proof}[\textbf{Proof}]
	We prove the necessity part of the theorem by contraposition. Suppose that $X$ is infinite. Then there are two possibilities to consider.

	%	Suppose that $\overline{T}(X,Y)$ is semi-balanced. Suppose to the contrary that $X$ is infinite. There are two possibilities to consider.
	
	\vspace{0.5mm}	
	\noindent Case 1: Suppose $X\setminus Y$ is finite. Then $Y$ is infinite. Therefore there exists a map $\alpha\colon Y \to Y$ that is injective but not surjective. Define a map $f\colon X\to X$ by
	\begin{align*}
		xf=
		\begin{cases}
			x\alpha &\text{if $x\in Y$}\\
			x & \text{if $x\in X\setminus Y$.}
		\end{cases}
	\end{align*}
	It is routine to verify that $f\in \overline{T}(X,Y)$. However, since $\alpha$ is injective but not surjective, we see that $f$ is injective but not surjective.
	
	\vspace{0.5mm}	
	\noindent Case 2: Suppose $X\setminus Y$ is infinite. Then there exists a map $\beta \colon X\setminus Y \to X\setminus Y$ that is injective but not surjective. Define a map $f\colon X\to X$ by
	\begin{align*}
		xf=
		\begin{cases}
			x & \text{if $x\in Y$}\\
			x\beta & \text{if $x\in X\setminus Y$.}
		\end{cases}
	\end{align*}
	It is routine to verify that $f\in \overline{T}(X,Y)$. However, since $\beta$ is injective but not surjective, we see that $f$ is injective but not surjective.
	
	\vspace{0.5mm}
	Thus, in either case, there exists a map $f\in \overline{T}(X,Y)$ that is injective but not surjective. Therefore $\mbox{c}(f)=0$ but $\mbox{d}(f)\geq 1$. This shows that $f$ is not semi-balanced. Hence $\overline{T}(X,Y)$ is not semi-balanced.
	
	%This shows that $f$ is not semi-balanced, a contradiction of our hypothesis that $\overline{T}(X,Y)$ is semi-balanced. Hence $X$ is finite. 
	
	\vspace{0.5mm}
	Conversely, suppose that $X$ is finite. Then $T(X)$ is semi-balanced by Theorem \ref{FTS-semibal-finite}. Therefore $\overline{T}(X,Y)$ is semi-balanced by Remark \ref{subsemigrp=semi-balanced}.
\end{proof}

%%%%%%%%%%%%%%%%%%%%%%%%%%%%
%%%%%%%%%%%%%%%%%%%%%%%%%%%%
%%%%%%%%%%%%%%%%%%%%%%%%%%%%
%%%%%%%%%%%%%%%%%%%%%%%%%%%%

\begin{remark}
	Let $X$ be a finite set. Then $\overline{T}(X,Y)$ is semi-balanced by Theorem \ref{T-XY-semibal-finite}. However, the semigroup $\overline{T}(X,Y)$ is not necessarily unit-regular (cf. Theorem \ref{unit-reg-TXY}).
\end{remark}

%%%%%%%%%%%%%%%%%%%%%%%%%%%%
%%%%%%%%%%%%%%%%%%%%%%%%%%%%
%%%%%%%%%%%%%%%%%%%%%%%%%%%%
%%%%%%%%%%%%%%%%%%%%%%%%%%%%
From Proposition \ref{unit-semibal}, we know that every unit-regular element in a submonoid $M$ of $T(X)$ is also a semi-balanced element of $M$. Moreover, we note from Theorem \ref{ureg-smbal-equi} that the concepts of unit-regular and semi-balanced elements are equivalent in $T(X)$. However, these concepts are not necessarily equivalent in $L(V)$ as seen from the following example. 

\begin{example}
	Let $V\colonequals\mathbb{Q}[x]$ be the vector space of all polynomials in $x$ over the field $\mathbb{Q}$. Clearly $\{1,x,x^2,\ldots\}$ is a basis for $V$. Consider $f\in L(V)$ such that 
	\begin{eqnarray*}
		p(x)f=
		\begin{cases}
			0     &  \text{if $p(x)=1$}\\
			xp(x) &  \text{if $p(x)\in \{x,x^2,x^3\ldots\}$}.
		\end{cases}
	\end{eqnarray*} 
	It is routine to verify that $N(f)=\langle 1\rangle=\mathbb{Q}$, $R(f)=\langle x^2,x^3,x^4,\ldots \rangle$, and $V/R(f)=\langle 1+R(f), x+R(f)\rangle$. Then $\nul(f)=1$ and $\corank(f)=2$. Therefore $f$ is not unit-regular in $L(V)$ by Corollary \ref{nul-corank-unit-reg-ele}. 
	
	\vspace{0.5mm} 
	Now, we prove that $f$ is semi-balanced. Since $|\mathbb{Q}|=\aleph_0$, we see that $|V|=\aleph_0$ and $|V/R(f)|=\aleph_0$. To see $f$ is semi-balanced, we first prove that $\textnormal{d}(f)=\aleph_0$. For this, fix $u\in V\setminus R(f)$ and define a map $\varphi \colon V\setminus R(f)\to V/R(f)$ by
	\begin{eqnarray*}
		v\varphi=
		\begin{cases}
			R(f)   &  \text{if $v=u$}\\
			v+R(f) &   \text{otherwise}.
		\end{cases}	
	\end{eqnarray*}
	It is clear that $\varphi$ is surjective, since $(u+v)\varphi=u+R(f)$ for all nonzero vector $v$ of $R(f)$. This yields $|V\setminus R(f)|\geq |V/R(f)|$. Therefore, since $|V/R(f)|=\aleph_0$, we get $\textnormal{d}(f)=\aleph_0$. Next, we prove that $\textnormal{c}(f)=\aleph_0$. For this, let $T_f$ be a transversal of the equivalence relation $\ker(f)$. Then we may write $\pi(f)=\{p(x)+N(f)\colon p(x)\in T_f\}$. Note for every $p(x)\in T_f$ that  $|p(x)+N(f)|=|N(f)|=\aleph_0$, so $|p(x)+N(f)\setminus \{p(x)\}|=\aleph_0$. Since $|T_f|\leq |V|$ and
	$V\setminus T_f=\displaystyle\bigcup_{p(x)\in T_f}\big(p(x)+N(f)\setminus \{p(x)\}\big)$,
	we get 
	\[\textnormal{c}(f)=\sum_{p(x)\in T_f}|p(x)+N(f)\setminus \{p(x)\}|=|T_f|\times \aleph_0=\aleph_0.\]
	Thus $\textnormal{c}(f)=\textnormal{d}(f)$ and consequently $f$ is semi-balanced.
	
\end{example}

%%%%%%%%%%%%%%%%%%%%%%%%%%%%
%%%%%%%%%%%%%%%%%%%%%%%%%%%%
%%%%%%%%%%%%%%%%%%%%%%%%%%%%
%%%%%%%%%%%%%%%%%%%%%%%%%%%%

%\begin{example}
%Let $V=\mathbb{R}[x]$, the set of all polynomials in $x$ over the field $\mathbb{R}$. Clearly $\{1,x,x^2,\ldots\}$ is a basis for $V$. Consider $f\in L(V)$ such that 
%	\begin{eqnarray*}
	%		p(x)f=
	%		\begin{cases}
		%			0     &  \text{if $p(x)=1$}\\
		%			xp(x) &  \text{if $p(x)\in \{x,x^2,x^3\ldots\}$}.
		%		\end{cases}
	%	\end{eqnarray*} 

%	
%\vspace{0.5mm} 
%Now, we prove that $f$ is semi-balanced. Notice first that $|V\setminus R(f)|\geq |V/R(f)|$ and $|V/R(f)|=\aleph_1$, where $\aleph_1$ denotes the cardinality of $\mathbb{R}$. Therefore $\textnormal{d}(f)=\aleph_1$. Let $T_f$ be a transversal of the equivalence relation $\ker(f)$ on $V$. Then we may write $\pi(f)=\{p(x)+N(f)\colon p(x)\in T_f\}$. Note that $|N(f)|==\aleph_1$ and $|p(x)+N(f)|=|N(f)|$ for all $p(x)\in T_f$, so $|p(x)+N(f)\setminus \{p(x)\}|=\aleph_1$ for all $p(x)\in T_f$. Since 
%\[V\setminus T_f=\bigcup_{p(x)\in T_f}\Big(p(x)+N(f)\setminus \{p(x)\}\Big),\]
%and $|T_f|\leq |V|=\aleph_1$, we get $\textnormal{c}(f)=\sum_{p(x)\in T_f}|p(x)+N(f)\setminus \{p(x)\}|=|T_f|\times \aleph_1=\aleph_1$. Thus $\textnormal{c}(f)=\textnormal{d}(f)$, so $f$ is semi-balanced.
%\end{example}

%%%%%%%%%%%%%%%%%%%%%%%%%%%%
%%%%%%%%%%%%%%%%%%%%%%%%%%%%
%%%%%%%%%%%%%%%%%%%%%%%%%%%%
%%%%%%%%%%%%%%%%%%%%%%%%%%%%

In the next theorem, we give a necessary and sufficient condition for the semigroup $L(V)$ to be semi-balanced.

\begin{theorem}\label{L-V-semi}
	The semigroup $L(V)$ is semi-balanced if and only if $V$ is finite-dimensional.
\end{theorem}

\begin{proof}[\textbf{Proof}]
	We prove the necessity part of the theorem by contraposition. Suppose that $V$ is infinite-dimensional. Let $B$ be a basis for $V$. Then, since $B$ is infinite, there exists a map $\alpha \colon B\to B$ that is injective but not surjective. Define a linear map $f\colon V \to V$ by $vf=v\alpha$ for all $v\in B$. It is clear that $f$ is injective but not surjective, since $\alpha$ is injective but not surjective. This gives $\textnormal{c}(f)=0$ and $\textnormal{d}(f)\ge 1$. Therefore $f$ is not semi-balanced. Hence $L(V)$ is not semi-balanced.
	
	\vspace{0.2mm}
	Conversely, suppose that $V$ is finite-dimensional. Then $L(V)$ is unit-regular by combining Proposition \ref{alarsf80-unit-facto} and Theorem \ref{facto-fin-dim}. Therefore $L(V)$ is semi-balanced by Corollary \ref{unit-reg--semi-balanced}.
\end{proof}

%%%%%%%%%%%%%%%%%%%%%%%%%%%%
%%%%%%%%%%%%%%%%%%%%%%%%%%%%
%%%%%%%%%%%%%%%%%%%%%%%%%%%%
%%%%%%%%%%%%%%%%%%%%%%%%%%%%
Thus, we have the following from Theorems \ref{unit-reg-fin-dim} and \ref{L-V-semi}
\begin{theorem}
	The following statements are equivalent.
	\begin{enumerate}
		\item[\rm(i)] The semigroup $L(V)$ is unit-regular.
		\item[\rm(ii)] The semigroup $L(V)$ is semi-balanced.
		\item[\rm(iii)] The space $V$ is finite-dimensional.	
	\end{enumerate}
\end{theorem}

%%%%%%%%%%%%%%%%%%%%%%%%%%%%
%%%%%%%%%%%%%%%%%%%%%%%%%%%%
%%%%%%%%%%%%%%%%%%%%%%%%%%%%
%%%%%%%%%%%%%%%%%%%%%%%%%%%%
In the next theorem, we give a necessary and sufficient condition for the semigroup $\overline{L}(V,W)$ to be semi-balanced.

\begin{theorem}\label{semibal-LVW}
	The semigroup $\overline{L}(V,W)$ is semi-balanced if and only if $V$ is finite-dimensional.
\end{theorem}
\begin{proof}[\textbf{Proof}]
	We prove the necessity part of the theorem by contraposition. Suppose that $V$ is infinite-dimensional. Then there are two possibilities to consider.
	
	%	Suppose that $\overline{L}(V,W)$ is semi-balanced. Suppose to the contrary that $V$ is infinite-dimensional. There are two possibilities to consider.
	
	\vspace{0.5mm}	
	\noindent Case 1: Suppose $V/W$ is finite-dimensional. Then $W$ is infinite-dimensional. Therefore there exists a linear map $\alpha\colon W \to W$ that is injective but not surjective. Define a map $f\colon V\to V$ by
	\begin{align*}
		vf=
		\begin{cases}
			v\alpha  & \text{if $v\in W$}\\
			v        & \text{if $v\in V/W$.}
		\end{cases}
	\end{align*}
	It is routine to verify that $f\in \overline{L}(V,W)$. However, since $\alpha$ is injective but not surjective, we see that $f$ is injective but not surjective.
	
	\vspace{0.5mm}	
	\noindent Case 2: Suppose $V/W$ is infinite-dimensional. Then there exists a linear map $\beta \colon V/W \to V/W$ that is injective but not surjective. Define a map $f\colon V\to V$ by
	\begin{align*}
		vf=
		\begin{cases}
			v       & \text{if $v\in W$}\\
			v' & \text{if $v\in V\setminus W$, where $(v+W)\beta=v'+W$.}
		\end{cases}
	\end{align*}
	It is routine to verify that $f\in \overline{L}(V,W)$. However, since $\beta$ is injective but not surjective, we see that $f$ is injective but not surjective.
	
	\vspace{0.5mm}
	Thus, in either case, there exists a map $f\in \overline{L}(V,W)$ that is injective but not surjective. This gives $\textnormal{c}(f)=0$ but $\textnormal{d}(f)\geq 1$. Therefore $f$ is not semi-balanced. Hence $\overline{L}(V,W)$ is not semi-balanced.
	
	\vspace{0.2mm}
	Conversely, suppose that $V$ is finite-dimensional. Then $L(V)$ is semi-balanced by Theorem \ref{L-V-semi}. Therefore $\overline{L}(V,W)$ is semi-balanced by Remark \ref{subsemigrp=semi-balanced}.
\end{proof}

%%%%%%%%%%%%%%%%%%%%%%%%%%%%%%%
%%%%%%%%%%%%%%%%%%%%%%%%%%%%%%%
%%%%%%%%%%%%%%%%%%%%%%%%%%%%%%%
%%%%%%%%%%%%%%%%%%%%%%%%%%%%%%%

\begin{remark}
	Let $V$ be a finite-dimensional vector space. Then $\overline{L}(V,W)$ is semi-balanced by Theorem \ref{semibal-LVW}. However, the semigroup $\overline{L}(V,W)$ is not necessarily unit-regular (cf. Theorem \ref{L-V-W=unit-reg}).
\end{remark}

%%%%%%%%%%%%%%%%%%%%%%%%%%%%
%%%%%%%%%%%%%%%%%%%%%%%%%%%%
%%%%%%%%%%%%%%%%%%%%%%%%%%%%
%%%%%%%%%%%%%%%%%%%%%%%%%%%%

\end{document}